\documentclass[12pt, reqno]{amsart}

\theoremstyle{definition}
\newtheorem{dfn}{Definition}[section]

\theoremstyle{remark}
\newtheorem{rem}[dfn]{Remark}

\theoremstyle{plain}
\newtheorem{thm}[dfn]{Theorem}
\newtheorem{lem}[dfn]{Lemma}
\newtheorem{prop}[dfn]{Proposition}

\newcommand{\Hom}{\mathop{\mathrm{Hom}}\nolimits}
\newcommand{\End}{\mathop{\mathrm{End}}\nolimits}
\newcommand{\Rad}{\mathop{\mathrm{Rad}}\nolimits}
\newcommand{\SLF}{\mathop{\mathrm{SLF}}\nolimits}
\newcommand{\soc}{{\rm Soc}}
\newcommand{\tr}{\operatorname{tr}}

\numberwithin{equation}{section}

\title[Symmetric linear functions and pseudotrace maps]
{Some remarks on symmetric linear functions and pseudotrace maps}
\author[Y.~Arike]{Yusuke Arike}
\address{Department of Pure and Applied Mathematics, Graduate School of Information Science and Technology, 
Osaka University, Toyonaka, Osaka 560-0043, JAPAN}
\email{y-arike@cr.math.sci.osaka-u.ac.jp}
\date{\today}

\begin{document}
\maketitle
\begin{abstract}
Let $A$ be a finite-dimensional associative algebra
and $\phi$ a symmetric linear function on $A$.
In this note, we will show that the pseudotrace maps defined in \cite{Mi}
are obtained as special cases of well-known symmetric linear functions on the 
 endomorphism rings of projective modules.
We also prove that modules are interlocked with $\phi$ 
if and only if they are projective.
As an application of our approach, we will give proofs of several propositions 
 and theorems in \cite{Mi}
for an arbitrary finite-dimensional associative algebra. 
\end{abstract}
\section{Introduction}
In this note, we work on an algebraically closed field $\mathbf{k}$ of characteristic $0$.
Let $A$ be a finite-dimensional associative $\mathbf{k}$-algebra.
A linear function $\phi$ on $A$ is said to be {\em symmetric}
if $\phi(ab) = \phi(ba)$ for all $a, b \in A$.
We denote the space of symmetric linear functions on $A$
by $\SLF(A)$.

In \cite{Mi}, Miyamoto introduces a notion of a {\em pseudotrace map} on a basic symmetric $\mathbf{k}$-algebra $P$
in order to construct pseudotrace functions of logarithmic modules of vertex operator algebras
satisfying some  finiteness condition called $C_2$-condition.
Let $\phi$ be a symmetric linear function on $P$ 
which induces a nondegenerate bilinear form $P \times P \to \mathbf{k}$.
Then 
the pseudotrace map $\tr_{W}^{\phi}$ is a symmetric linear function
on the endomorphism ring of a finite-dimensional right $P$-module $W$ 
called {\em interlocked with $\phi$}.
As it is implicitly mentioned in \cite{Mi} and 
it is proved in this note, a finite-dimensional right $P$-module which is interlocked with $\phi$
is in fact a direct sum of indecomposable projective modules.

For an arbitrary finite-dimensional $\mathbf{k}$-algebra $A$,
a finitely generated projective right $A$-module $W$ has an {\em $A$-coordinate system of $W$},
that is, $\{u_i\}_{i=1}^{n} \subset W$ and $\{\alpha_i\}_{i=1}^{n} \subset \Hom_A(W, A)$
such that $w = \sum_{i=1}^{n} u_i \alpha_i(w)$ for all $w \in W$ (see \cite{B}).
For any symmetric linear function $\phi$ on $A$,
we can define a symmetric linear function on $\End_A(W)$ by
\begin{align}
\phi_{W}(\alpha) = \phi (\sum_{i=1}^{n}\alpha_i \circ \alpha (u_i)) \notag
\end{align}
for all $\alpha \in \End_A(W)$ (c.f. \cite{Br}).
In this note, we show that the symmetric linear function $\tr_W^{\phi}$ coincides with the pseudotrace map 
when $A = P$ and $\phi$ induces a nondegenerate symmetric associative bilinear form on $P$.
We also prove that a right $P$-module $W$ is interlocked with $\phi$
if and only if $W$ is projective.
Then we can prove several propositions and theorems in \cite{Mi}
for {\em arbitrary finite-dimensional $\mathbf{k}$-algebras}.

This note is organized as follows.
In section 2, we recall a construction of a symmetric linear function $\phi_{W}$ 
on the endomorphism ring of finitely generated 
projective modules $W$ from a symmetric linear function $\phi$ on $A$.
In section 3, 
we assume that $P$ is indecomposable, basic and symmetric
and $\phi \in \SLF(P)$ satisfies some conditions (see section 3).
We recall a notion of a right $P$-module $W$ which is interlocked with $\phi$ 
and a notion of a pseudotrace map $\tr_{W}^{\phi}$ defined in \cite{Mi}.
We show that $W$ is interlocked with $\phi$ if and only if $W$ is projective. 
By using this fact, for any indecomposable projective module $W$,
we define $\phi_{W}$ and show that $\phi_{W}$ coincides with $\tr_{W}^{\phi}$.
In section 4 and 5, we prove several propositions and theorems for pseudotrace maps in \cite{Mi} by using 
$\phi_{W}$ for arbitrary finite-dimensional $\mathbf{k}$-algebras. 

\section{Projective modules and symmetric linear functions}
Let $A$ be a finite-dimensional associative $\mathbf{k}$-algebra.
We denote a left (resp. right) $A$-module $M$ by $_AM$ (resp. $M_A$).

In this section, we recall a notion of a symmetric linear function
on the endomorphism ring of a finitely generated projective right $A$-module (c.f. \cite{Br}).

Assume that  $W_A$ is finitely generated.
Then $W_A$ is projective if and only if
there exist subsets $\{u_i\}_{i=1}^n \subset W_A$
and $\{\alpha_i\}_{i=1}^n \subset \Hom_{A}(W_A, A)$ such that
\begin{align}
w = \sum_{i=1}^{n} u_i \alpha_i(w) \notag
\end{align}
for all $w \in W_A$ (see \cite{B}, chapter II, \S 2.6, Proposition 12).
The set $\{u_i, \alpha_i\}_{i=1}^{n}$ is called
an {\em $A$-coordinate system} of $W_A$. 

Assume that $W_A$ is finitely generated and projective.
Let $\{u_i, \alpha_i\}_{i=1}^{n}$ be an $A$-coordinate system of $W_A$.
Then we define a  map 
\begin{align}
T_{W_A} : \End_A(W_A) \to A / [A, A], \notag
\end{align}
by 
$\alpha \mapsto \pi \left(\sum_{i=1}^{n} \alpha_i \circ 
 \alpha (u_i) \right)$
where $\pi : A \to A/[A, A]$ is the canonical surjection (c.f. \cite{H}, \cite{S}).
It is known that the map $T_{W_A}$ does not depend on the choice of 
$A$-coordinate systems and that
$T_{W_A} (\alpha \circ \beta) = T_{W_A} (\beta \circ \alpha)$
for all $\alpha, \beta \in \End_A(W_A)$ (see \cite{H}, \cite{S}).
For $\phi \in \SLF(A)$, we set $\phi_{W_A} = \phi \circ T_{W_A} : \End_A(W_A) \to \mathbf{k}$.
Then we have the following.

\begin{prop}\label{p-1.2.1}
Assume that $W_A$ is finitely generated and projective and
let $\phi$ be a symmetric linear function on $A$.
Then $\phi_{W_A}$ is a symmetric linear function on $\End_A(W_A)$.
\end{prop}

\section{Miyamoto's psedotrace maps}
In this section, we show that the map $\phi_{W_A}$ coincides with the
pseudotrace map defined in \cite{Mi} if $A$ satisfies extra conditions.

First we recall the definition of a pseudotrace map.
Let $P$ be a basic symmetric indecomposable $\mathbf{k}$-algebra
and assume that  $\phi \in \SLF(P)$  induces a nondegenerate
symmetric associative bilinear form $\langle \ , \ \rangle : P \times P \to \mathbf{k}$.
We fix a decomposition of the unity $1$ by mutually orthogonal primitive idempotents:
\begin{align}
1 = e_1 + e_2 + \cdots + e_k. \notag
\end{align}
We further assume that $\phi (e_i) = 0$ for all $1 \le i \le k$.
Note that we have
$P / J(P) = \mathbf{k} \bar{e}_1 \oplus \cdots \oplus \mathbf{k} \bar{e}_k$
since $P$ is basic and indecomposable.
It is well-known that $\{e_i P | 1 \le i \le k\}$
is the complete list of indecomposable projective right $P$-modules.

Since $a \in \soc (P_P)$ if and only if $a J(P) = 0$
we see that
\begin{align}
\langle a J(P), P  \rangle 
= \langle J(P), a \rangle = \langle P, J(P) a \rangle  = 0. \notag
\end{align}
The same argument for $\soc({}_PP)$ shows
$\soc(P_P) = \soc({}_PP)$.
Thus $\soc(P_P) = \soc({}_PP)$ is a two-sided ideal 
and we denote it by $\soc(P)$.
Then we have $\langle a J(P), P \rangle = \langle a, J(P) \rangle$
for any $a \in P$.
This identity shows that 
$\soc(P) = J(P)^{\perp}$.
Similarly we have $J(P) = \soc(P)^{\perp}$.
Thus the bilinear form $\langle \ , \ \rangle$ induces a
nondegenerate pairing $\langle \ , \ \rangle : \soc(P) \times P / J(P) \to \mathbf{k}$.
Let $\{f_1, f_2, \dots , f_k \}$ be a basis of $\soc(P)$ which are dual to the basis
$\{\bar{e}_1, \bar{e}_2, \dots , \bar{e}_k\}$ of $P/J(P)$, that is,
$\langle f_i, \bar{e}_j \rangle = \langle f_i, e_j \rangle= \delta_{ij}$
for $\ 1 \le i, j \le k$.

\begin{lem}\label{p-3.2.1}
$e_i f_j = f_j e_i = \delta_{ij} f_j$ for all $1 \le i, j \le k$.
\end{lem}
\begin{proof}
Note that $e_i f_j \in \soc (P)$.
Thus we have
\begin{align}
\langle e_i f_j, \bar{e}_k \rangle =  \delta_{ik} \langle f_j, \bar{e}_k \rangle = \delta_{ik} 
 \delta_{kj} \notag
\end{align}
so that $e_i f_j = \delta_{ij} f_j$.
\end{proof}

\begin{lem}\label{p-3.2.2}
$\soc(P) \subseteq J(P)$,
in particular, $e_i \soc(P) e_j \subseteq e_i J(P) e_j$ for all $1 \le i, j \le k$.
\end{lem}
\begin{proof}
Since $P = \bigoplus_{i=1}^{k} P e_i$, we see that
$\soc(P) = \bigoplus_{i=1}^{k} \soc(Pe_i)$.
Then $J(P)e_i$ is the unique maximal submodule of $Pe_i$.
Suppose that $\soc(Pe_i)$ is not contained in $J(P)$.
We have
$Pe_i = \soc(Pe_i)+J(P)e_i$
since $J(P)e_i$ is the unique maximal submodule of $P e_i$.
Then we conclude $\soc(Pe_i) = Pe_i$ by Nakayama's lemma.
Therefore we can see  $e_i \in \soc(Pe_i)$.
By the same argument for $P=\bigoplus_{i=1}^{k} e_i P$,
we obtain $e_i \in \soc(e_iP)$.
Thus we find $J(P) e_i = e_i J(P) = 0$, which shows that $e_i$
is a central idempotent of $P$.
This contradicts to the assumption that $P$ is indecomposable.
\end{proof}

Since $P = \sum_{i=1}^{k} \mathbf{k} e_i + J(P)$,
we have by Lemma \ref{p-3.2.1}
\begin{align}
e_i P e_j = \begin{cases}
\mathbf{k} e_i + e_i J(P) e_i, & i=j, \\
e_i J(P) e_j, & i \not= j, 
\end{cases}\label{eq-3.0.1}
\end{align}
and
\begin{align}
e_i \soc(P) e_j = \begin{cases}
\mathbf{k} f_i, & i=j, \\
0, & i \not= j.
\end{cases}\label{eq-3.0.2}
\end{align}
Set $d_{ij} = \dim_{\mathbf{k}} e_i J(P) e_j / e_i \soc(P) e_j$ for all $1 \le i, j \le k$.
Then since the pairing 
\begin{align}
\langle \ , \ \rangle : e_i J(P) e_j / e_i \soc(P) e_j  \times &  e_j J(P) e_i / 
 e_j \soc(P) e_i 
\to \mathbf{k} \notag
\end{align}
is well-defined and nondegenerate, it follows that $d_{ij} = d_{ji}$ for all $1 \le i, j \le k$.

Also since
$e_i \soc(P) e_j \subseteq e_i J(P) e_j \subseteq e_i P e_j$, \eqref{eq-3.0.1}
and \eqref{eq-3.0.2},
we have $\dim_{\mathbf{k}} e_i P e_i = d_{ii} + 2$ and $\dim_{\mathbf{k}} e_i P e_j = d_{ij}$
for $i \not= j$
by \eqref{eq-3.0.1} and \eqref{eq-3.0.2}.

\begin{lem}[\cite{Mi}, Lemma 3.2]\label{p-3.1}
The algebra $P$ has a basis
\begin{align}
\Omega = \{\rho_0^{ii}, \rho_{d_{ii}+1}^{ii}, \rho_{s_{ij}}^{ij} | 
1 \le i, j \le k, \ 1 \le s_{ij} \le d_{ij} \} \notag
\end{align}
satisfying
\begin{enumerate}
\item\label{p-3.1.a}
$\rho_0^{ii} = e_i$, $\rho_{d_{ii} + 1}^{ii} = f_i$,
\item\label{p-3.1.b}
$e_i \rho_s^{ij} e_j = \rho_s^{ij}$,
\item\label{p-3.1.c}
$\langle \rho_s^{ij}, \rho_{d_{ab} + 1 - t}^{ab} \rangle 
= \delta_{i, b} \delta_{j, b} \delta_{s, t}$,
\item\label{p-3.1.d}
$\rho_s^{ij} \rho_{d_{ji} + 1 - s}^{ji} = f_i$,
\item\label{p-3.1.e}
the space spanned by $\{\rho_t^{ij} | t \ge s\}$ is $e_i P e_i$-invariant.
\end{enumerate}
\end{lem}

For $1 \le i, j \le k$, set
\begin{align}
&\Omega_i = \{\rho_s^{ij} | 1 \le j \le k, s \}, \ \Omega_{ij} = \{\rho_s^{ij} | s\}. \notag
\end{align}
Note that $\Omega_i$ is a basis of $e_i P$ for any $1 \le i \le k$
and  $\Omega - \{e_1, \dots , e_k\}$ is a basis of $J(P)$.
We sometimes denote an element of $\Omega_{ij}$ by $\rho^{ij}$.

\begin{dfn}[\cite{Mi}, Definition 3.6]
Assume that $W_P$ is finitely generated.
The module $W_P$ is said to be {\em interlocked with $\phi$}
if 
$\ker (f_i) = \{w \in W | w f_i = 0\}$
is equal to $\sum_{\rho \in \Omega - \{e_i \}} W \rho$
for all $1 \le i \le k$.
\end{dfn}
It is obvious that $\ker(f_i) \supseteq \sum_{\rho \in \Omega - \{e_i \}} W \rho$
since $\rho f_i = 0$ for any $\rho \in \Omega - \{e_i \}$.
In \cite{Mi}, the pseudotrace map is defined on the endomorphism ring of 
a finite-dimensional right $P$-module which is interlocked with $\phi$.
The isomorphism stated in \cite[p.68]{Mi} is more precisely understood as follows.
\begin{thm}\label{p-3.1.1}
Assume that $W_P$ is finitely generated.
Then $W_P$ is interlocked with $\phi$
if and only if $W_P$ is projective.
In particular, if $W_P$ is interlocked with $\phi$
then the multiplicity of the indecomposable projective module $e_i P$
in $W_P$ is given by $\dim_{\mathbf{k}} W_P f_i$ for $1 \le i \le k$.
\end{thm}
In order to prove this theorem, we first show the following lemmas.
\begin{lem}\label{p-3.1.2}
Any indecomposable projective module $e_i P$ for $1 \le i \le k$
is interlocked with $\phi$.
\end{lem}
\begin{proof}
For $e_i p \in e_i P$,
suppose $e_i p f_i = 0$ and
express $p$ as $p = \sum_{\rho \in \Omega} a_{\rho} \rho$ with $a_{\rho} \in \mathbf{k}$.
Then 
$0 = e_i p f_i = e_i \sum_{\rho \in \Omega} a_{\rho} \rho f_i = a_{e_i} f_i$.
Thus $p$ belongs to the space spanned by $\Omega - \{e_i\}$,
which shows $e_i p \in \sum_{\rho \in \Omega - \{e_i\}} e_i P \rho$.

For $i \not= j$,
we can see $e_i p f_j = a_{e_j} e_i f_j = 0$ for all $p \in P$.
Thus we have $\ker(f_j) \subseteq e_i P = \sum_{\rho \in \Omega - \{e_j \}} e_i P \rho$.
\end{proof}
\begin{lem}\label{p-3.1.3}
The module $W_P$ is interlocked with $\phi$ if and only if any direct summand of $W_P$
is interlocked with $\phi$.
\end{lem}
\begin{proof}
Suppose that $W_P = W_1 \oplus W_2$ where $W_1$ and $W_2$ are right $P$-modules.
Then we have
\begin{align}
\sum_{\rho \in \Omega - \{e_i\}} W \rho = (\sum_{\rho \in \Omega - \{e_i\}}
W_1 \rho ) \oplus (\sum_{\rho \in \Omega - \{e_i\}} W_2 \rho). \label{eq-3.1.1}
\end{align}
If $W_P$ is interlocked with $\phi$ and $w f_i = 0$ for $w \in W_1$,
we see $w \in \sum_{\rho \in \Omega - \{e_i\}}
W_1 \rho$ by \eqref{eq-3.1.1}.
Similarly $W_2$ is  interlocked with $\phi$.

Conversely, suppose that $W_1$ and $W_2$ are interlocked with $\phi$
and $(w_1+w_2) f_i = 0$ for $w_1 \in W_1$ and $w_2 \in W_2$.
Then
we have $w_1 f_i = 0$ and $w_2 f_i = 0$,
which shows that $w_1 + w_2 \in \sum_{\rho \in \Omega - \{e_i\}} W \rho$
by \eqref{eq-3.1.1}.
Therefore we conclude that $W_P$ is interlocked with $\phi$.
\end{proof}

\begin{lem}\label{p-3.1.4}
Assume that $W_P$ is interlocked with $\phi$.
Then
\begin{align}
W e_i / W J(P) e_i \cong W f_i, \  \overline{w e_i} \mapsto wf_i \notag
\end{align}
for any $1 \le i \le k$. 
\end{lem}
\begin{proof}
The kernel of the map
$W e_i \to W f_i, \ we_i \mapsto wf_i$
is equal to $\sum_{\rho \in \Omega - \{e_i\}} W \rho e_i = W J(P) e_i$
since $W_P$ is interlocked with $\phi$.
\end{proof}

\begin{proof}[Proof of Theorem \ref{p-3.1.1}]
By Lemma \ref{p-3.1.2} and Lemma \ref{p-3.1.3},
any finite direct sum of indecomposable projective modules is interlocked with $\phi$.

Conversely, suppose that $W_P$ is interlocked with $\phi$.
By Lemma \ref{p-3.1.4}, there exists $v^{e_i}$
such that $v^{e_i} f_i \not= 0$ if $\dim_{\mathbf{k}} W f_i \not=0$.
Then the map
\begin{align}
\theta : e_i P \to W, \ e_i p \mapsto v^{e_i} e_i p, \notag
\end{align}
is a $P$-homomorphism.
Suppose $\ker(\theta) \not=0$.
Note that $\soc (e_i P) = \mathbf{k} f_i$
by Lemma \ref{p-3.2.1}.
Since $e_i P$ has the unique simple submodule $\soc (e_i P)$
(see \cite[Proposition 9.9 (ii)]{CR})
we have $f_i \in \ker (\theta)$ and $v^{e_i} f_i = 0$.
This is a contradiction.
Thus $\theta$ is injective.

Since $P$ is a symmetric algebra, any projective module is also injective
(see \cite[Proposition 9.9 (iii)]{CR}).
Therefore $\theta$ is split and then $e_i P$ is a direct summand of $W$,
say, $W \cong e_iP \oplus W^{\prime}$.
By Lemma \ref{p-3.1.1}, $W^{\prime}$ is also interlocked with $\phi$
and $\dim_{\mathbf{k}} W^{\prime} f_i = \dim_{\mathbf{k}} Wf_i - 1$ since $\dim_{\mathbf{k}} e_i P f_i = 1$.
If $W f_i = 0$ for all $1 \le i \le k$,
then $w f_i = 0$ for all $w \in W$ and $1 \le i \le k$.
Thus we have
\begin{align}
W = \bigcap_{i=1}^{k} \left(\sum_{\rho \in \Omega - \{e_i\}} W \rho \right) = W J(P). \notag
\end{align}
By Nakayama's lemma, we have $W = 0$.

Therefore the induction on $\dim_{\mathbf{k}} W f_i$ proves
 the theorem.
In particular, the multiplicity of $e_i P$ in $W$
is equal to $\dim_{\mathbf{k}} W f_i$ for all $1 \le i \le k$.
\end{proof}

Assume that $W_P$ is finitely generated and projective.
Then $W_P$ is isomorphic to a finite direct sum of indecomposable projective modules:
\begin{align}
W_P \cong \bigoplus_{i=1}^{k} n_i e_i P, \label{eq-3.0}
\end{align}
where $n_i$ is the multiplicity of $e_i P$, that is, $n_i = \dim_{\mathbf{k}} W f_i$.
We denote the element of $W_P$ corresponding to  $e_i$
by $v_j^{e_i}$ for $1 \le i \le k$ and $1 \le j \le n_i$.
Note that $W_P$ has a basis 
$\{v_j^{e_i} \rho \ | \ \rho \in \Omega_i, \ 1 \le i \le k, \ 1 \le j \le n_i\}$. 

Since 
$\alpha (v_j^{e_i}) = \alpha (v_j^{e_u} e_i) = \alpha (v_j^{e_i}) e_i \in W e_i$
for $\alpha \in \End_P(W_P)$ and Lemma \ref{p-3.1} (b), 
we have
\begin{align}
\alpha (v_j^{e_i}) = \sum_{s=1}^{k} \sum_{t=1}^{n_s} \sum_{\rho^{si} \in \Omega_{si}} \alpha_{jt}^{\rho^{si}}
v_t^{e_s} \rho^{si} \label{eq-3.1}
\end{align}
for $1 \le i \le k$ and $1 \le j \le n_i$
where $\alpha_{jt}^{\rho^{si}} \in \mathbf{k}$.
In \cite{Mi},
the pseudotrace map $\tr_{W_P}^{\phi}$ on $\End_P(W_P)$ is defined by
\begin{align}
\tr_{W_P}^{\phi} (\alpha) = \sum_{i=1}^{k}\sum_{j=1}^{n_i} \alpha_{jj}^{f_i}. \label{eq-3.2}
\end{align}

In order to show that the pseudotrace map coincides with $\phi_{W_P}$,
we choose the following $P$-coordinate system of $W_P$.
Note that $\phi_{W_P}$ does not depend on the choice of $P$-coordinate systems.

Set
\begin{align}
\alpha_j^{i} (v_t^{e_s} \rho^{s p}) = \begin{cases}
\rho^{ip}, & i=s, \ j=t, \\
0, & \text{otherwise}, 
\end{cases}\notag
\end{align}
for $1 \le i \le k$ and $1 \le j \le n_i$.
Then $\alpha_{j}^i$ belongs to $\Hom_P(W_P, P)$ for $1 \le i \le k$ and $1 \le j \le n_i$.
\begin{lem}\label{p-3.2}
The set
$\{v_j^{e_i}, \alpha_j^i \ | \ 1 \le i \le k, \ 1 \le j \le n_i\}$
is a $P$-coordinate system of $W_P$.
\end{lem}
\begin{proof}
By the definitions of  $v_j^{e_i}$
and  $\alpha_{j}^i$, we have  
$v_j^{e_i} \rho^{ip} 
= v_j^{e_i} \alpha_j^i (v_j^{e_i} \rho^{ip}) 
= \sum_{s=1}^{k}\sum_{t=1}^{n_s} v_t^{e_s} \alpha_t^s(v_j^{e_i} \rho^{ip})$.
Since the elements $v_j^{e_i} \rho^{ip}$
form a basis of $W_P$, we have shown the lemma.
\end{proof}

\begin{thm}\label{p-3.3.3.1}
Assume that
$W_P$ is finitely generated and projective.
Then $\phi_{W_P} = \tr_{W_P}^{\phi}$.
\end{thm}
\begin{proof}
For $\alpha \in \End_P(W_P)$,
one has
\begin{align}
\phi_{W_P}(\alpha) &= \phi \left(\sum_{i=1}^{k} \sum_{j=1}^{n_i}
\alpha_j^{i} \circ \alpha (v_j^{e_i}) \right) \notag\\
&= \phi \left(\sum_{i, s=1}^{k} \sum_{j=1}^{n_i} \sum_{t=1}^{n_s} \sum_{\rho^{si}\in \Omega_{si}}
\alpha_j^i (\alpha_{jt}^{\rho^{si}} v_{t}^{e_s} \rho^{si})\right) \notag\\
&= \sum_{i=1}^{k} \sum_{j=1}^{n_i} \sum_{\rho^{ii} \in \Omega_{ii}} \alpha_{jj}^{\rho^{ii}}
\phi (\rho^{ii}) \notag\\
&= \sum_{i=1}^{k} \sum_{j=1}^{n_i} \alpha_{jj}^{f_i} \notag\\
&= \tr_{W_P}^{\phi} (\alpha), \notag
\end{align}
since \eqref{eq-3.1} and Lemma \ref{p-3.1} \eqref{p-3.1.c}.
\end{proof}

\section{The center and symmetric linear functions}

In this section, we assume that the finite-dimensional 
$\mathbf{k}$-algebra $A$ contains a nonzero central element $\nu$ such that
$(\nu - r)^s A = 0$ and $(\nu - r)^{s-1} A \not=0$ for some $r \in \mathbf{k}$
and $s \in \mathbb{Z}_{> 0}$.

Set $\mathcal{K} = \{a \in A \ | \ (\nu - r) a = 0\}$.
Note that $\mathcal{K}$ is a two-sided ideal of $A$.
Let $\alpha : M_A \to N_A$ be an $A$-module homomorphism.
Then $M / M \mathcal{K}$ is an $A / \mathcal{K}$-module
and the map $\widehat{\alpha} : M / M \mathcal{K} \to N / N \mathcal{K}$
defined by $\widehat{\alpha}(\overline{m}) = \overline{\alpha(m)}$
is an $A / \mathcal{K}$-module homomorphism
where $\overline{m}$ is the image of $m$ under the canonical map $M \to M / M\mathcal{K}$.
Assume that $W_A$ is finitely generated and projective and
let $\{u_i, \alpha_i\}_{i=1}^{n}$ be an $A$-coordinate system of $W_A$.
Then $\{\overline{u}_i, \widehat{\alpha}_i\}_{i=1}^{n}$
is an $A/ \mathcal{K}$-coordinate system of the right $A / \mathcal{K}$-module 
$W / W\mathcal{K}$.

Let $\phi$ be a symmetric linear function on $A$.
Then  $\phi^{\prime} (\overline{a}) = \phi ((\nu - r) a)$ for any $\overline{a} 
\in A / \mathcal{K}$ is well-defined and symmetric on $A / \mathcal{K}$.

\begin{prop}[c.f. \cite{Mi}, Proposition 3.8]\label{p-1.3.0}
Assume that $W_A$ is finitely generated and projective.
Let $\phi$ be a symmetric linear function on $A$.
Then
\begin{align}
\phi_{W_A} (\alpha \circ (\nu - r)) = 
 \phi^{\prime}_{W/W\mathcal{K}}(\widehat{\alpha}) \notag
\end{align}
for all $\alpha \in \End_A(W_A)$ where $\nu - r$ is identified as an element of $\End_A(W_A)$.
\end{prop}
\begin{proof}
Let $\{u_i, \alpha_i\}_{i=1}^{n}$ be an $A$-coordinate system of $W_A$.
Then we have
\begin{align}
\phi^{\prime}_{W/W\mathcal{K}} (\widehat{\alpha})
&= \phi^{\prime}\left(\sum_{i=1}^{n} \widehat{\alpha}_i \circ \widehat{\alpha} 
 (\overline{u}_i)\right) \notag\\
&= \phi \left((\nu-r) \sum_{i=1}^{n} \alpha_i \circ \alpha (u_i)\right) \notag\\
&= \phi \left(\sum_{i=1}^{n} \alpha_i \circ \alpha (u_i (\nu - r))\right) \notag\\
&= \phi_{W_A}(\alpha \circ (\nu - r)). \notag
\end{align}
\end{proof}

\section{Basic algebras and symmetric linear functions}
Let 
\begin{align}
1 = \sum_{i=1}^{n} \sum_{j = 1}^{n_i} e_{ij} \label{eq:1.4.1}
\end{align}
be a decomposition of the unity $1$ by mutually orthogonal primitive idempotents
where $e_{ij} A \cong e_{ik} A$ and
$e_{ij} A \not\cong e_{k \ell} A$ for $i \not=k$.
Set $e_i = e_{i1}$ for $1 \le i \le n$
and $e = \sum_{i=1}^{n} e_i$.
Then $\mathbf{k}$-algebra $eAe$ with the unity $e$ is called a {\em basic algebra} associated
with $A$.
Then $Ae$ is $(A, eAe)$-bimodule.
Let
$\ell : A \to \End_{eAe} (Ae_{eAe})$ and 
$r : eAe \to \End_A ({}_AAe)$
be maps defined by
$\ell (a) (be) = abe$ for all $a, b \in A$ and 
$r (eae) (be) = beae$ for all $a, b \in A$. 

\begin{lem}[\cite{AF}, Proposition 4.15,  Theorem 17.8]\label{p-1.4.2}
\begin{enumerate}
\item\label{p-1.4.2.a}
The map $r$
is an anti-isomorphism of algebras.
\item\label{p-1.4.2.b}
The map $\ell$ is an isomorphism
of algebras
\end{enumerate}
\end{lem}
By Lemma \ref{p-1.4.2}, an element $a \in A$ is identified as an element in $\End_{eAe}(Ae)$
and an element $eae \in eAe$ is identified as an element in $\End_{A}(Ae)$. 

\begin{rem}
By Lemma \ref{p-1.4.2}, we have two linear maps
\begin{align}
&(-)_{Ae_{eAe}} : \SLF(eAe) \to \SLF(A), \
(-)_{{}_AAe} : \SLF(A) \to \SLF((eAe)^{\operatorname{op}}). \notag
\end{align}
Since $\SLF(eAe) = \SLF((eAe)^{\operatorname{op}})$,
the second map is in fact a map $\SLF(A) \to \SLF(eAe)$.
\end{rem}

By \eqref{eq:1.4.1}, we have 
\begin{align}
A e = \bigoplus_{i=1}^{n} \bigoplus_{j=1}^{n_i} e_{ij} A e. \label{eq:1.4.2}
\end{align}
The following fact is well-known.

\begin{lem}\label{p-1.4.0}
Let $e$ and $f$ be idempotents of $A$.
Then the following assertions are equivalent.
\begin{enumerate}
\item\label{p-1.4.0.a}
$A e \cong A f$.
\item\label{p-1.4.b}
$e A \cong f A$.
\item\label{p-1.4.0.c}
There exist $p \in e A f$ and $q \in f A e$ such that
$pq = e$ and $qp = f$.
\end{enumerate}
\end{lem}
%
%

\begin{lem}\label{p-1.4.1}
For $1 \le i \le n$ and $1 \le j \le n_i$, we have
$e_i A e \cong e_{ij} A e$ as
right $e A e$-modules. 
\end{lem}
\begin{proof}
By Lemma \ref{p-1.4.0} and the fact $e_i A \cong e_{ij} A$,
there exist $p_{ij} \in e_{ij} A e_i$ and $q_{ij} \in e_i A e_{ij}$
such that $p_{ij} q_{ij} = e_{ij}$ and $q_{ij} p_{ij} = e_i$.
Then the maps $\alpha : e_{ij} A e \to e_{i} A e$ defined by
$\alpha (e_{ij} a e) = q_{ij} a e$ and
$\beta : e_{i} A e \to e_{ij} A e$ defined by $\beta (e_i a e) = p_{ij} a e$
are $e A e$-homomorphisms and are inverse each other.
Thus we have shown the lemma.
\end{proof}

For any $a e \in Ae$,
it is not difficult to check that $ae = \sum_{i = 1}^{n} \alpha_i(a) e_i$
where $\alpha_i (a) = a e_i$.
Thus $\{e_i, \alpha_i \}_{i=1}^{n}$ is an $A$-coordinate system of $_{A}Ae$.

By the proof of  Lemma \ref{p-1.4.2},
we can see that
$e_{ij} A e_{eAe}$ is generated by $p_{ij} \in e_{ij} A e_i$
such that $p_{ij} q_{ij} = e_{ij}$ and $q_{ij} p_{ij} = e_i$ for some $q_{ij} 
\in e_i A e_{ij}$.
Note that we can choose $p_{i1} = q_{i1} = e_{i1} = e_i$.
For any $a e \in Ae$, we set $\beta_{ij} (ea) = q_{ij} a e \in eAe$
for all $1 \le i \le n$ and $1 \le j \le n_i$.
Then we have $\beta_{ij} \in \Hom_{eAe} (A e, eAe)$
and 
$\sum_{i = 1}^{n} \sum_{j=1}^{n_i} p_{ij} \beta_{ij} (a e) = 
\sum_{i = 1}^{n} \sum_{j=1}^{n_i} e_{ij} (a e) = ae$
by \eqref{eq:1.4.1}.
Thus $\{p_{ij}, \beta_{ij} | 1 \le i \le n, \ 1 \le j \le n_i \}$
is an $eAe$-coordinate system of $Ae_{eAe}$.
In the following, we fix the $A$-coordinate system $\{e_i, \alpha_i\}_{i=1}^n$ 
of $_AAe$ and the $eAe$-coordinate system $\{p_{ij}, \beta_{ij} | 1 \le 
i \le n, \ 1 \le j \le n_i \}$ of $Ae_{eAe}$.
\begin{lem}\label{p-1.4.3}
\begin{enumerate}
\item\label{p-1.4.3.a}
Let $\phi$ be a symmetric linear function on $A$. 
Then 
$\phi_{_A Ae} (eae) = \phi (e a e)$
for all $eae \in eAe$.
\item\label{p-1.4.3.b}
Let $\psi$ be a symmetric linear function on $eAe$. 
Then $\psi_{Ae_{eAe}} (a) = \psi (\sum_{i=1}^{n} \sum_{j=1}^{n_i}
q_{ij} a p_{ij})$
for all  $a \in A$.
\end{enumerate}
\end{lem}
\begin{proof}
Since 
$\phi_{_AAe} (eae) = \sum_{i=1}^{n} \phi (\alpha_i (e_i eae))
= \sum_{i=1}^{n} \phi (e_i a e_i)$ and
$\phi$ is symmetric, we obtain
$\phi (e_i A e_j) = \phi (e_j e_i A e_j) = 0$ for $i \not= j$,
which shows the first assertion.

The second assertion is proved as follows:
\begin{align}
\psi_{Ae_{eAe}} (a) &= \sum_{i = 1}^{n} \sum_{j=1}^{n_i} \psi 
 (\beta_{ij} (a p_{ij})) = \sum_{i = 1}^{n} \sum_{j=1}^{n_i} \psi 
 ( q_{ij} a p_{ij}). \notag
\end{align}
\end{proof}

\begin{thm}\label{p-1.4.4}
\begin{enumerate}
\item\label{p-1.4.4.a}
Let $\phi$ be a symmetric linear function on $A$.
Then
$(\phi_{_AAe})_{Ae_{eAe}}(a) = \phi (a)$
for all $a \in A$.
\item\label{p-1.4.4.b}
Let $\psi$ be a symmetric linear function on $eAe$.
Then we have
$(\psi_{Ae_{eAe}})_{_AAe} (eae) = \psi (eae)$
for all $eae \in eAe$.
\item\label{p-1.4.4.c}
The space of symmetric linear functions on $A$
and the one of $eAe$ are isomorphic as vector spaces.
\end{enumerate}
\end{thm}
\begin{proof}
By Lemma \ref{p-1.4.3}, we have
\begin{align}
(\phi_{_AAe})_{Ae_{eAe}}(a) &= (\phi)_{_AAe} 
(\sum_{i=1}^{n} \sum_{j=1}^{n_i} q_{ij} a p_{ij}) 
= \phi (\sum_{i=1}^{n} \sum_{j=1}^{n_i} q_{ij} a p_{ij} ) \notag\\ 
&= \phi (\sum_{i=1}^{n} \sum_{j=1}^{n_i} p_{ij} q_{ij} a) 
= \phi (\sum_{i=1}^{n} \sum_{j=1}^{n_i} e_{ij} a) = \phi (a) \notag
\end{align}
which shows \eqref{p-1.4.4.a}.

By Lemma \ref{p-1.4.3}, we have
\begin{align}
(\psi_{Ae_{eAe}})_{_AAe}(eae)
&= \psi_{Ae_{eAe}} (eae) 
= \psi (\sum_{i=1}^{n} \sum_{j=1}^{n_i} q_{ij} eae p_{ij}) \notag\\
&= \psi (eae), \notag
\end{align}
since $q_{i1} = p_{i1} = e_{i}$.

Hence we can see that two linear maps $(-)_{_AAe} : \SLF(A) \to \SLF(eAe)$
and $(-)_{Ae_{eAe}} : \SLF(eAe) \to \SLF(A)$
are inverse each other,
which shows the last assertion.
\end{proof}

\begin{rem}
The statement \eqref{p-1.4.4.a} of Theorem \ref{p-1.4.4} for $a \in \soc(A)$
is found in \cite[Lemma 3.9]{Mi}.
The statement \eqref{p-1.4.4.c} of Theorem \ref{p-1.4.4} is well-known (see \cite[6.1]{NS}).
\end{rem}
For $\phi \in \SLF(A)$,
we set
$\Rad(\phi) = \{a \in A \ | \ \phi(A a) = 0\}$.
Then $\Rad(\phi)$ is a two-sided ideal of $A$
and $\phi$ induces a symmetric linear function on $A / \Rad(\phi)$.
Note that $A / \Rad(\phi)$ is a symmetric algebra
since $\phi$ is well-defined on $A / \Rad(\phi)$
and induces a nondegenerate symmetric associative bilinear form
on $A / \Rad(\phi)$.

Let $A = A_1 \oplus A_2 \oplus \cdots \oplus A_{\ell}$
be a decomposition into two-sided ideals of $A$.
For any $\phi \in \SLF(A)$, 
we have $\phi = \phi_1 + \phi_2 + \cdots + \phi_{\ell}$
where $\phi_i = \phi | _{A_{i}}$.
Note that $\phi_{i} \in \SLF(A_i)$.
If $\phi (a A) = 0$ for some  $a \in A$,
then we can see that $\phi_i (a A_{i}) \subseteq \phi (a A) = 0$.

\begin{thm}\label{p-1.4.6}
Let
$\phi$ be a symmetric linear function on $A$
and $\nu$ a central element of $A$.
Assume that $\phi ((\nu - r)^{s} a) = 0$ for any $a \in A$
and that $A = A_1 \oplus A_2 \oplus \cdots \oplus A_{\ell}$
is a decomposition of $A$ into two-sided ideals.
Then there exist symmetric linear functions $\phi_i \in \SLF(A_i)$,
basic symmetric algebras $P_i$ of $B_i = A / \Rad (\phi_i)$
and $(A, P_i)$-bimodules $M_i$
satisfying $(\nu - r)^s M_i= 0$.
Moreover, 
\begin{align} 
\phi (b) = \sum_{i=1}^{\ell} ((\phi_{i})_{{}_{B_i} M_i})_{(M_i)_{P_i}} (b) \notag
\end{align}
for all $b \in A$
where $b$ in the right hand side is viewed as a linear
map defined by the left action of $b \in A$ on each $(A, P_i)$-bimodule $M_i$. 
\end{thm}
\begin{proof}
Set $B_i = A / \Rad(\phi_i)$.
Since  $\Rad(\phi_i) \supseteq A_j$ for $j \not= i$,
we can see that $B_i = A_i / \Rad(\phi_i)$.
We first note that the symmetric linear function $\phi_i$ on $B_i$
is well-defined
and that $B_i$ is naturally a left $A$-module.
Let $P_i = \overline{e}_i (A / \Rad(\phi_i)) \overline{e}_i$ be
the basic algebra of $A / \Rad(\phi_i)$ where $\overline{e}_i$ is an idempotent of
$B_i$.
The basic algebra $P_i$ is a symmetric algebra by \cite[10.1]{NS}.
Then we set  $M_i = (A / \Rad(\phi_i)) \overline{e}_i$
which is an  $(A, P_i)$-bimodule.
By the argument before the statement of this theorem, 
we can see that $(\nu - r)^s \in \Rad(\phi_i)$ 
and
thus $(\nu - r)^s M_i = 0$.
Note that the left action of $a \in A$
defines a right $P_i$-module endomorphism of $M_i$.
By Lemma \ref{p-1.4.3},
we have
$\phi_i (b) = \phi_i (\overline{b}) = ((\phi_i)_{_{B_i} M_i})_{(M_i)_{P_i}} (\overline{b})
= ((\phi_i)_{_{B_i} M_i})_{(M_i)_{P_i}}(b)$
for all $b \in A_i$, which shows the theorem.
\end{proof}

\begin{rem}
This theorem is found in \cite[Theorem 3.10]{Mi}.
In the proof of \cite[Theorem 3.10]{Mi}, 
it is shown that a symmetric linear function on $A$
may be written as a sum of pseudotrace maps even if $A$ is indecomposable 
by using the fact
$(\phi_{_A Ae})_{Ae_{eAe}} (a) = \phi (a)$ for all $a \in \soc(A)$ (see \cite[Lemma 3.9]{Mi})
in our notation.
However, since $(\phi_{_A Ae})_{Ae_{eAe}} (a) = \phi (a)$ for all $a \in A$,
any symmetric linear function can be written by
only one symmetric linear function on the endomorphism ring of the 
$(A, P)$-bimodule
if $A$ is indecomposable.
\end{rem}
\section*{Acknowledgement}
The author expresses his gratitude to Prof.~K.~Nagatomo
for his encouragement and useful comments.

\end{document}